\documentclass[reqno]{amsart}
\usepackage{amsmath}
\usepackage{amssymb}
\usepackage{amsthm}
\usepackage{mathrsfs}
\usepackage{url}
\newtheorem{definition}{Definition}[section]
\newtheorem{theorem}[definition]{Theorem}
\newtheorem{lemma}[definition]{Lemma}
\newtheorem{prop}[definition]{Proposition}

\newtheorem{remark}[definition]{Remark}

\usepackage{color}

\begin{document}

\title{A Polynomial Approximation Result for Free Herglotz-Agler Functions}
\author{Kenta Kojin}
\address{
Graduate School of Mathematics, Nagoya University, 
Furocho, Chikusaku, Nagoya, 464-8602, Japan
}
\email{m20016y@math.nagoya-u.ac.jp}
\date{\today}
\keywords{Schwarz lemma, nc Schur-Agler calss, free Herglotz-Agler class}
\thanks{The author would like to take this opportunity to thank the “Nagoya University
Interdisciplinary Frontier Fellowship” supported by JST and Nagoya University.
} 

\begin{abstract}
In this paper, we prove a noncommutative (nc for short) analog of Schwarz lemma for the nc Schur-Agler class and prove that the regular nc Schur-Agler class and the regular free Herglotz-Agler class are homeomorphic. Moreover, we give a characterization of regular free Herglotz-Agler functions. As an application, we will show that any regular free Herglotz-Agler functions can uniformly be approximated by regular Herglotz-Agler free polynomials.
\end{abstract}

 \maketitle


\section{Introduction}

In the context of complex analysis, a holomorphic function from the open unit disk into the right half plane is called a Herglotz function and has been studied in detail. A Herglotz function is said to be regular if it maps $0$ to $1$. About a century ago, Herglotz \cite{Her} showed that a regular Herglotz function admits an integral representation with a probability measure on the unit circle $\mathbb{T}$.


\begin{theorem}$($Herglotz \cite{Her}$)$
A holomorphic function $h$ defined on the open unit disk $\mathbb{D}$ is a regular Herglotz function, that is $\Re h\ge0$ and $h(0)=1$ if and only if there exists a unique probability measure $\mu$ supported on $\mathbb{T}$ such that
\begin{equation*}
h(x)=\int_{\mathbb{T}} \frac{1+e^{i\theta}x}{1-e^{i\theta}x} \, d\mu(e^{i\theta}).
\end{equation*}
\end{theorem}


Agler \cite{Agl} developed this representation theory for an appropriate class in the setting of several variables based on operator theory. Recently, Pascoe, Passer and Tully-Doyle \cite{PPT} proved an nc analog of Herglotz representation theorem.
Motivated by these works, we will give a polynomial approximation type characterization of regular free Herglotz-Agler functions. The consequence is that any regular free Herglotz-Agler functions on the polynomial polyhedron $B_{\delta}$ associated with a matrix $\delta$ of free polynomials that satisfies $\delta(0)=0$ can uniformly be approximated by regular Herglotz-Agler free polynomials on each $K_{\delta,r}=\{x\in B_{\delta}\;|\;\|\delta(x)\|\le r\}$ with $0<r<1$.  We have known a polynomial approximation result for nc Schur-Agler functions \cite[Theorem 3.3]{Koj}, which assert that every regular nc Schur-Agler function can uniformly be approximated by regular nc Schur-Agler free polynomials on each $K_{\delta, r}$. Hence it is natural to translate the approximation sequence $\{p_n\}_{n=1}^{\infty}$ into the regular free Herglotz-Agler class via the Cayley transform $f\mapsto\displaystyle\frac{1+f}{1-f}$. However, we encounter a topological problem, namely, we must estimate $\displaystyle\frac{1}{\|1-p_n(x)\|}$ uniformly on each $K_{\delta, r}$. We will overcome this problem by proving an nc analog of Schwarz lemma, which is one of the main observations of this paper.

The celebrated Schwarz lemma in the classical complex analysis asserts that a bounded holomorphic function $f$ on $\mathbb{D}$ with norm condition
 $\displaystyle\sup_{x\in\mathbb{D}}|f(x)|\le 1$ and $f(0)=0$ must satisfy $|f(x)|\le |x|$ for all 
$x\in\mathbb{D}$ and $|f'(0)|\le 1$ (see e.g., \cite[Theorem 12.2]{Rud}). We will generalize this to the setting of several noncommutative variables appropriately in section 3. This may be regarded as a generalization of a part of Popescu's work \cite[Corollary 2.5]{Pop2006}. In fact, what we will actually prove is that
\begin{equation*}
\|f(x)\|\le\|\delta(x)\|
\end{equation*}
holds for every regular nc Schur-Agler function $f$. We will crucially use this fact later. For example, it enables us  to prove an nc analog of maximum principle in section 3. Its proof was inspired by Popescu's observation in \cite[Theorem 2.5]{Pop2010}. We will also use it to show that the regular nc Schur-Agler class and the regular free Herglotz-Agler class are homeomorphic to each other via the Cayley transforms. Hence we can translate the previous polynomial approximation result for nc Schur-Agler functions \cite[Theorem 3.3]{Koj} into the regular free Herglotz-Agler class. In this way, we will establish the main result of this paper.

In closing of the introduction, we emphasize that our Schwarz lemma type result is an nc analog of that for Schur-Agler functions rather than Schur functions.


\section{Preliminaries}

We review some materials on nc functions. Let $\mathbb{M}_n^d$ denote the set of $d$-tuples of $n\times n$ matrices and let $\mathbb{M}^d$ denote the {\bf $d$-dimensional noncommutative universe}, which is given by 

\begin{equation*}
\mathbb{M}^d=\displaystyle\coprod_{n=1}^{\infty}\mathbb{M}_n^d.
\end{equation*}
This is the domain of free polynomials. There are a couple of natural operations on $\mathbb{M}^d$. If $x\in\mathbb{M}_m^d$ and $y\in\mathbb{M}_n^d$, then
\begin{equation*}
x\oplus y:=\left(\left[
\begin{matrix}
x^1&0\\
0&y^1
\end{matrix}
\right],\ldots,\left[
\begin{matrix}
x^d&0\\
0&y^d
\end{matrix}
\right]\right)\in\mathbb{M}_{m+n}^d.
\end{equation*}
If $x\in\mathbb{M}_n^d$, $\alpha$ is a $k\times n$ matrix and $\beta$ is an $n\times m$ matrix, then 
\begin{equation*}
\alpha x\beta:=(\alpha x^1\beta,\ldots, \alpha x^d\beta).
\end{equation*}
 We define the $C^*$ norm on each set $\mathbb{M}_n^d$ by the formula
\begin{equation*}
\|x\|_n:=\displaystyle\max_{1\le r\le d}\|x^r\|_{B(\mathbb{C}^n)}.
\end{equation*}
We say that a set $\Omega\subset\mathbb{M}^d$ is an {\bf nc set} if $\Omega$ is closed under the direct sums  i.e., if $x\in\Omega_n:=\Omega\cap\mathbb{M}_n^d$ and $y\in\Omega_m$, then $x\oplus y\in \Omega_{n+m}$. An nc set $\Omega\subset\mathbb{M}^d$ is an {\bf nc domain} if $\Omega$ is {\bf disjoint union open}, which means that $\Omega\cap\mathbb{M}_n^d$ is Euclidean open for all $n\ge 1$.


A function $f:\Omega\rightarrow\mathbb{M}$ is an {\bf nc function} if
\begin{enumerate}
\item $f$ is {\bf graded}, i.e., if $x\in\Omega_n$, then $f(x)\in\mathbb{M}_n$, and
\item $f$ {\bf respects intertwinings}, i.e., whenever $x\in\Omega_n$, $y\in\Omega_m$ and an $m\times n$ matrix $\alpha$ satisfy $\alpha x=y\alpha$, then $\alpha f(x)=f(y)\alpha$.
\end{enumerate}
Note that a function $f$ on an nc subset is nc if and only if $f$ satisfies the following conditions (see \cite[Section I.2.3]{KVV}):
\begin{enumerate}
\item $f$ is graded,
\item $f$ {\bf respects direct sums}, i.e., if $x$ and $y$ are in $\Omega$, then $f(x\oplus y)=f(x)\oplus f(y)$, and
\item $f$ {\bf respects similarities}, i.e., whenever $x, y\in\Omega_n$,
 $\alpha\in\mathbb{M}_n$ with $\alpha$ invertible such that
  $y=\alpha x\alpha^{-1}$, then $f(y)=\alpha f(x)\alpha^{-1}$.
\end{enumerate}


Next, we will define the free topology, the nc Schur-Agler class and the free Herglotz-Agler class. Let $\delta$ be an $s\times r$ matrix of free polynomials in $d$-variables. Let
\begin{equation*}
B_{\delta}=\{x\in\mathbb{M}^d\;|\;\|\delta(x)\|<1\}.
\end{equation*}
Here $\|\delta(x)\|$ denotes the operator norm. Then, $B_{\delta}$ becomes an nc domain. A set of the above form is called a {\bf polynomial polyhedron}. The free topology is the topology on $\mathbb{M}^d$ generated by all polynomial polyhedra. The {\bf nc Schur-Agler class on $B_{\delta}$}, $\mathcal{SA}(B_{\delta})$, is defined by
\begin{equation*}
\mathcal{SA}(B_{\delta}):=\left\{f:B_{\delta}\rightarrow\mathbb{M}\;\middle|\;\mbox{$f$ is nc and} \displaystyle\sup_{x\in B_{\delta}}\|f(x)\|\le 1\;\right\}.
\end{equation*}
In addition, we assume $\delta(0)=0$. A function in the nc Schur-Agler class on $B_{\delta}$ is regular if $f(0)=0$. We denote by $\mathcal{RSA}(B_{\delta})$ the set of functions in the nc Schur-Agler class that is regular. We call this class the {\bf regular nc Schur-Agler class on $B_{\delta}$}.

Agler and McCarthy \cite{AMg} showed that each function in the nc Schur-Agler class admits a realization formula. Ball, Marx and Vinnikov \cite{BMV} studied the nc Schur-Agler class in a more general setting. 


\begin{theorem}$($\cite[Corollary 8.13]{AMg}$)$
Let $B_{\delta}$ be a polynomial  polyhedron, and let $f$ be a graded function from $B_{\delta}$ into $\mathbb{M}$. Then, the following conditions are equivalent:
\begin{enumerate}
\item $f\in\mathcal{SA}(B_{\delta})$.
\item There exist an auxiliary Hilbert space $\mathcal{X}$ and a unitary operator
\begin{equation*}
U=
\begin{bmatrix}
A&B\\
C&D
\end{bmatrix}:
\begin{bmatrix}
\mathcal{X}\otimes\mathbb{C}^s\\
\mathbb{C}
\end{bmatrix}\rightarrow
\begin{bmatrix}
\mathcal{X}\otimes\mathbb{C}^r\\
\mathbb{C}
\end{bmatrix}
\end{equation*}
such that for all $x\in(B_{\delta})_n$, 
\begin{equation*}
f(x)=
\begin{matrix}
D\\
\otimes\\
I_n
\end{matrix}
+
\begin{matrix}
C\\
\otimes\\
I_n
\end{matrix}\left(
\begin{matrix}
I_{\mathcal{X}}\\
\otimes\\
I_{n\times s}
\end{matrix}
-
\begin{matrix}
I_{\mathcal{X}}\\
\otimes\\
\delta(x)
\end{matrix}
\begin{matrix}
A\\
\otimes\\
I_n
\end{matrix}
\right)^{-1}
\begin{matrix}
I_{\mathcal{X}}\\
\otimes\\
\delta(x)
\end{matrix}
\begin{matrix}
B\\
\otimes\\
I_n
\end{matrix}\;\;.
\end{equation*}
\end{enumerate}
\end{theorem}

In this paper, we only treat the regular nc Schur-Agler class. We will crucially use a realization formula to prove an nc analog of Schwarz lemma and the maximum principle in the next section.


Finally, we define the regular free Herglotz-Agler class.  Let $\delta$ be a matrix of free polynomials in $d$-variables with $\delta(0)=0$. The {\bf regular free Herglotz-Agler class on $B_{\delta}$}, $\mathrm{RHA}(B_{\delta})$, is defined by

\begin{equation*}
\mathrm{RHA}(B_{\delta}):=\left\{h:B_{\delta}\rightarrow\mathbb{M}\middle|\; h \;\mbox{is nc}, \Re{h(x)}=\frac{h(x)+h(x)^*}{2}\ge 0\; \mbox{and}\; h(0)=I\right\}.
\end{equation*}

We endow $\mathcal{RSA}(B_{\delta})$ and $\mathrm{RHA}(B_{\delta})$ with the topology of uniform convergence on closed polynomial polyhedra. Namely, a net of functions
 $\{f_{\lambda}\}_{\Lambda}$ converges to $f$ if and only if for every 
 $K_{\delta,r}=\{x\in B_{\delta}\;|\;\|\delta(x)\|\le r\}$, $\{f_{\lambda}\}_{\Lambda}$ uniformly norm-converges to $f$ on $K_{\delta,r}$. The topology is first countable. We will give an explicit relation between the regular nc Schur-Agler class and the free Herglotz-Agler class in the next section.


\section{Nc Schwarz lemma and regular nc Schur-Agler class v.s. regular free Herglotz-Agler class}

First, we recall the {\bf right difference-differential operator} $\Delta$ and the holomorphy of nc functions. Let $\Omega\subset\mathbb{M}^d$ be an nc domain. Since $\Omega$ is right admissible i.e., if $x\in\Omega_n$, $y\in\Omega_m$ and $z$ is a $d$-tuples of $n\times m$ matrices, then there exists a non-zero complex number $t$ such that 
$\begin{bmatrix}
x& tz\\
0& y
\end{bmatrix}\in \Omega_{n+m}$. Then, for any nc functions $f$, there exists an $n\times m$ matrix $w$ so that
\begin{equation*}
f\left(
\begin{bmatrix}
x& tz\\
0& y
\end{bmatrix}\right)=
\begin{bmatrix}
f(x)&w\\
0& f(y)
\end{bmatrix},
\end{equation*}
and we define $\Delta f(x,y)(z):=t^{-1}w$. See \cite[Proposition 2.2]{KVV}. Then $\Delta f(x,y)$ gives a linear map from the space of all $d$-tuples of $n\times m$ matrices to the space of all $n\times m$ matrices (\cite[Proposition 2.4 and Proposition 2.6]{KVV}). For nc functions, local boundedness and holomorphy are equivalent (\cite[Theorem 12.17]{AMY}, \cite[Corollary 7.6]{KVV}). 
In particular, an nc Schur-Agler function $f\in\mathcal{SA}(B_{\delta})$ is 
Fr\'{e}chet 
differentiable on each level $(B_{\delta})_n$ and its Fr\'{e}chet derivative at $x\in(B_{\delta})_n$ is given by the linear operator 
$\Delta f(x,x):\mathbb{M}_n^d\rightarrow\mathbb{M}_n$ (\cite[Theorem 7.2]{KVV}). The left difference-differential operator $\Delta_L$ is also available. By \cite[Proposition 2.8]{KVV}, $\Delta_R f(x,y)=\Delta_L f(y,x)$ on an nc set. Thus, it suffices to discuss only the right one.

In the rest of this paper, we will assume $\delta(0)=0$. In this paper, we will crucially use the next proposition, which should be understood as an nc analog of famous Schwarz lemma. Popescu \cite[Corollary 2.5]{Pop2006}  (essentially) showed the inequality in the next proposition when $\delta(x)=[x_1\cdots x_d]$. The method of the proof below may be known among specialists.


\begin{prop}\label{Schwarz}
Let $f\in\mathcal{RSA}(B_{\delta})$. Then, for any $x\in B_{\delta}$, we have
\begin{equation*}
\|f(x)\|\le \|\delta(x)\|.
\end{equation*}  
Moreover, if $B_{\delta}$ contains the noncommutative polydisc 
\begin{equation*}
\mathbb{D}_{nc}^d:=\{(x_1,\ldots, x_d)\;|\;\displaystyle\max_{1\le r\le d}\|x^r\|< 1\},
\end{equation*}
then the Fr\'{e}chet derivative of $f$ at $0\in (B_{\delta})_n$ must be contractive for all $n\ge 1$.
\end{prop}
\begin{proof}
For simplicity, the identity operator is always denoted by $I$ without indicating the matrix size, etc. Since $f$ is in $\mathcal{RSA}(B_{\delta})$, there exist an auxiliary Hilbert space $\mathcal{X}$ and a unitary operator
\begin{equation*}
U=
\begin{bmatrix}
A&B\\
C&0
\end{bmatrix}:
\begin{bmatrix}
\mathcal{X}\otimes\mathbb{C}^s\\
\mathbb{C}
\end{bmatrix}\rightarrow
\begin{bmatrix}
\mathcal{X}\otimes\mathbb{C}^r\\
\mathbb{C}
\end{bmatrix}
\end{equation*}
such that for all $x\in B_{\delta}$, 
\begin{equation*}
f(x)=
\begin{matrix}
C\\
\otimes\\
I
\end{matrix}\left(
I
-
\begin{matrix}
I\\
\otimes\\
\delta(x)
\end{matrix}
\begin{matrix}
A\\
\otimes\\
I
\end{matrix}
\right)^{-1}
\begin{matrix}
I\\
\otimes\\
\delta(x)
\end{matrix}
\begin{matrix}
B\\
\otimes\\
I
\end{matrix}\;\;.
\end{equation*}
Set $r:=\|\delta(x)\|$. Since $U$ is a unitary operator, we have
\begin{equation*}
\begin{bmatrix}
AA^*+BB^*&AC^*\\
CA^*&CC^*
\end{bmatrix}=
\begin{bmatrix}
I&0\\
0&I
\end{bmatrix}.
\end{equation*}
Since $CC^*=I$ and $BB^*=I-AA^*$, we obtain that
\begin{align*}
&r^2I-f(x)f(x)^*\\
&=r^2
\begin{matrix}
\;\;CC^*\\
\otimes\\
I
\end{matrix}-
\begin{matrix}
C\\
\otimes\\
I
\end{matrix}\left(I-
\begin{matrix}
I\\
\otimes\\
\delta(x)
\end{matrix}
\begin{matrix}
A\\
\otimes\\
I
\end{matrix}\right)^{-1}
\begin{matrix}
I\\
\otimes\\
\delta(x)
\end{matrix}
\begin{matrix}
\;\;BB^*\\
\otimes\\
I
\end{matrix}
\begin{matrix}
I\\
\otimes\\
\delta(x)^*
\end{matrix}\left(I-
\begin{matrix}
\;A^*\\
\otimes\\
I
\end{matrix}
\begin{matrix}
I\\
\otimes\\
\delta(x)^*
\end{matrix}\right)^{-1}
\begin{matrix}
\;C^*\\
\otimes\\
I
\end{matrix}\\
&=
\begin{matrix}
C\\
\otimes\\
I
\end{matrix}\left(I-
\begin{matrix}
I\\
\otimes\\
\delta(x)
\end{matrix}
\begin{matrix}
A\\
\otimes\\
I
\end{matrix}\right)^{-1}\left[r^2\left(I-
\begin{matrix}
I\\
\otimes\\
\delta(x)
\end{matrix}
\begin{matrix}
A\\
\otimes\\
I
\end{matrix}\right)\left(I-
\begin{matrix}
\;A^*\\
\otimes\\
I
\end{matrix}
\begin{matrix}
I\\
\otimes\\
\delta(x)^*
\end{matrix}\right)\right.\\
&\;\;\;\;\;\;\;\;\;\;\;\;\;\;\;\;\;\;\;\;\;\;\;\;\;\;\;\;\;\;\;\;\;\;\;\;\;\;\;\;\;-
\left.\begin{matrix}
I\\
\otimes\\
\delta(x)
\end{matrix}
\begin{matrix}
\;\;I-AA^*\\
\otimes\\
I
\end{matrix}
\begin{matrix}
I\\
\otimes\\
\delta(x)^*
\end{matrix}\right]
\left(I-
\begin{matrix}
\;A^*\\
\otimes\\
I
\end{matrix}
\begin{matrix}
I\\
\otimes\\
\delta(x)^*
\end{matrix}\right)^{-1}
\begin{matrix}
\;C^*\\
\otimes\\
I
\end{matrix}.
\end{align*}
As $CA^*=AC^*=0$, it follows that
\begin{align*}
&r^2I-f(x)f(x)^*\\
&=\begin{matrix}
C\\
\otimes\\
I
\end{matrix}\left(I-
\begin{matrix}
I\\
\otimes\\
\delta(x)
\end{matrix}
\begin{matrix}
A\\
\otimes\\
I
\end{matrix}\right)^{-1}\left[r^2\left(I-
\begin{matrix}
I\\
\otimes\\
\delta(x)
\end{matrix}
\begin{matrix}
A\\
\otimes\\
I
\end{matrix}\right)\left(I-
\begin{matrix}
\;A^*\\
\otimes\\
I
\end{matrix}
\begin{matrix}
I\\
\otimes\\
\delta(x)^*
\end{matrix}\right)+r^2\left(I-
\begin{matrix}
I\\
\otimes\\
\delta(x)
\end{matrix}
\begin{matrix}
A\\
\otimes\\
I
\end{matrix}\right)
\begin{matrix}
\;A^*\\
\otimes\\
I
\end{matrix}
\begin{matrix}
I\\
\otimes\\
\delta(x)^*
\end{matrix}\right.\\
&\left.+r^2
\begin{matrix}
I\\
\otimes\\
\delta(x)
\end{matrix}
\begin{matrix}
A\\
\otimes\\
I
\end{matrix}\left(I-
\begin{matrix}
A^*\\
\otimes\\
I
\end{matrix}
\begin{matrix}
I\\
\otimes\\
\delta(x)^*
\end{matrix}\right)-
\begin{matrix}
I\\
\otimes\\
\delta(x)
\end{matrix}
\begin{matrix}
\;\;I-AA^*\\
\otimes\\
I
\end{matrix}
\begin{matrix}
I\\
\otimes\\
\delta(x)^*
\end{matrix}\right]\left(I-
\begin{matrix}
\;A^*\\
\otimes\\
I
\end{matrix}
\begin{matrix}
I\\
\otimes\\
\delta(x)^*
\end{matrix}\right)^{-1}
\begin{matrix}
\;C^*\\
\otimes\\
I
\end{matrix}\\
&=\begin{matrix}
C\\
\otimes\\
I
\end{matrix}\left(I-
\begin{matrix}
I\\
\otimes\\
\delta(x)
\end{matrix}
\begin{matrix}
A\\
\otimes\\
I
\end{matrix}\right)^{-1}\left[r^2I-
\begin{matrix}
I\\
\otimes\\
\delta(x)
\end{matrix}
\begin{matrix}
I\\
\otimes\\
\delta(x)^*
\end{matrix}\right.\\
&\;\;\;\;\;\;\;\;\;\;\;\;\;\;\;\;\;\;\;\;\;\;\;\;\;\;\;\;\;\;\;\;\;\;\;\;\;+(1-r^2)
\left.\begin{matrix}
I\\
\otimes\\
\delta(x)
\end{matrix}
\begin{matrix}
\;\;AA^*\\
\otimes\\
I
\end{matrix}
\begin{matrix}
I\\
\otimes\\
\delta(x)^*
\end{matrix}\right]\left(I-
\begin{matrix}
\;A^*\\
\otimes\\
I
\end{matrix}
\begin{matrix}
I\\
\otimes\\
\delta(x)^*
\end{matrix}\right)^{-1}
\begin{matrix}
\;C^*\\
\otimes\\
I
\end{matrix}\ge 0.
\end{align*}
Hence we conclude that $\|f(x)\|\le \|\delta(x)\|$.

Next, we prove that the Fr\'{e}chet derivative $\Delta f(0,0)$ at $0$ must be contractive on each level. We may assume that $f\in\mathcal{RSA}(\mathbb{D}_{nc}^d)$. For any $n\ge 1$, and $z\in\mathbb{M}_n^d$, there exists $0<t<1$ such that $
\begin{bmatrix}
0& tz\\
0&0
\end{bmatrix}\in (\mathbb{D}_{nc}^d)_n$. Then, we have seen that
\begin{equation*}
\left\|
\begin{bmatrix}
f(x)&w\\
0&f(y)
\end{bmatrix}
\right\|_{2n}=
\left\| f\left(
\begin{bmatrix}
0& tz\\
0& 0
\end{bmatrix}
\right)\right\|_{2n}\le \left\|
\begin{bmatrix}
0& tz\\
0&0
\end{bmatrix}
\right\|_{2n}.
\end{equation*}
Then,
\begin{equation*}
\| w\|_n\le t\|z\|_n,
\end{equation*}
and hence
\begin{equation*}
\|\Delta f(0,0)(z)\|_n\le \|z\|_n,
\end{equation*}
since $\Delta f(0,0)(z)=t^{-1}w$. This means that the Fr\'{e}chet derivative at 0 is contractive.
\end{proof}


\begin{remark}
$(1)$ \upshape By the same calculation, we can prove an analog of Schwarz lemma for a function that admits a realization formula. Such examples are, the nc Schur-Agler class in the Ball, Marx and Vinnikov framework \cite{BMV}, the operator NC Schur-Agler class studied by Augat and McCarthy \cite{AuM}, the Schur-Agler class \cite{Agl, AT} (note that this class is nothing but the ``level 1" of the nc Schur-Agler class; see \cite[Theorem 8.19]{AMg}), and the contractive multipliers of an irreducible complete Pick Hilbert function space \cite{AMp}. In the last case, we have to calculate the zeros of the injection $b$ in \cite[Theorem 8.2]{AMp}.\\
$(2)$ Knese \cite{Kne} and Anderson, Dritschel and Rovnyak \cite{ADR} studied the part of the classical Schwarz lemma dealing with derivatives in several variables. Actually, the method of the above proof is the same as theirs.
\end{remark}


Popescu \cite[Theorem 5.1]{Pop2010} applied his analog of Schwarz lemma to proving a maximum principle for free holomorphic functions on the noncommutative ball in conjunction with the free automorphisms of the noncommutative ball. Next, we will prove an analogous fact in the present context. Note that we may not explicitly assume holomorphy in the next result. Its reason is that local boundedness and holomorphy are equivalent for nc functions.


\begin{theorem}\label{mmp}
Let $f$ be an nc function on $B_{\delta}$. If there exists an $x_0\in B_{\delta}$ such that
\begin{equation*}
\|f(x_0)\|\ge\|f(x)\|\;\;\mbox{for all} \;\;x\in B_{\delta},
\end{equation*}
then $f$ must be a constant nc function i.e., there exists a $c\in\mathbb{C}$ such that for all $n\ge 1$ and $x\in (B_{\delta})_n$, $f(x)=cI_n$.
\end{theorem}


\begin{proof}
Without loss of generality, we may assume that $\|f(x_0)\|=\displaystyle\sup_{x\in B_{\delta}}\|f(x)\|=1$. Then $f$ admits a realization formula. Namely, there exist an auxiliary Hilbert space $\mathcal{X}$ and a unitary operator
\begin{equation*}
U=
\begin{bmatrix}
A&B\\
C&D
\end{bmatrix}:
\begin{bmatrix}
\mathcal{X}\otimes\mathbb{C}^s\\
\mathbb{C}
\end{bmatrix}\rightarrow
\begin{bmatrix}
\mathcal{X}\otimes\mathbb{C}^r\\
\mathbb{C}
\end{bmatrix}
\end{equation*}
such that for all $x\in B_{\delta}$,
\begin{equation*}
f(x)=
\begin{matrix}
D\\
\otimes\\
I
\end{matrix}
+
\begin{matrix}
C\\
\otimes\\
I
\end{matrix}\left(
I
-
\begin{matrix}
I\\
\otimes\\
\delta(x)
\end{matrix}
\begin{matrix}
A\\
\otimes\\
I
\end{matrix}
\right)^{-1}
\begin{matrix}
I\\
\otimes\\
\delta(x)
\end{matrix}
\begin{matrix}
B\\
\otimes\\
I
\end{matrix}\;\;.
\end{equation*}
Since $f$ respects direct sums, $f(0_n)$ is determined only by $f(0_1)$ for all $n\ge 2$, where $0_n$ is the zero of $\mathbb{M}_n^d$.
If $\|f(0)\|=1$, then $|D|=1$ because $\delta(0)=0$. Since $U$ is a unitary operator, we have
\begin{equation*}
CC^*+DD^*=1_{\mathbb{C}}.
\end{equation*}
Hence, we get $C=0$. Therefore, the realization formula implies that $f$ must be a constant nc function. 

We then consider the case when $\|f(0)\|<1$. Set $\alpha I:=f(0)$, and we have $|\alpha|<1$. We define an nc function $\phi_{\alpha}$ on $\overline{\mathbb{D}_{nc}^1}:=\{w\in\mathbb{M}^1\;|\;\|w\|\le1\}$ by  the formula 
\begin{equation*}
\phi_{\alpha}(w):=(w-\alpha I)(I-\overline{\alpha} w)^{-1}. 
\end{equation*}
Then for all $w\in\overline{\mathbb{D}_{nc}^1}$, we have
\begin{align*}
I-\phi_{\alpha}(w)^*\phi_{\alpha}(w)&=I-(I-\alpha w^*)^{-1}(w^*-\overline{\alpha}I)(w-\alpha I)(I-\overline{\alpha} w)^{-1}\\
&=(I-\alpha w^*)^{-1}(I-\alpha w^*)(I-\overline{\alpha}w)(I-\overline{\alpha}w)^{-1}\\
&\;\;-(I-\alpha w^*)^{-1}(w^*-\overline{\alpha}I)(w-\alpha I)(I-\overline{\alpha} w)^{-1}\\
&=(1-|\alpha|^2)(I-\alpha w^*)^{-1}(I-ww^*)(I-\overline{\alpha}w)^{-1}\ge 0.  
\end{align*}
Hence, $\phi_{\alpha}(\overline{\mathbb{D}_{nc}^1})\subset\overline{\mathbb{D}_{nc}^1}$. Moreover, since $(I-\overline{\alpha} w)^{-1}$ is invertible, $\phi_{\alpha}$ maps a strict contraction to a strict contraction. Here, a strict contraction means an operator whose operator norm is less than one. In addition, we can easily see that $\phi_{\alpha}^{-1}=\phi_{-\alpha}$. Therefore,  $\|w\|=1$ if and only if $\|\phi_{\alpha}(w)\|=1$. (Note that this property can be regarded as a special case of \cite[Lemma 4.1]{Pop2010}). Set $g:=\phi_{\alpha}\circ f$. Since $g\in\mathcal{RSA}(B_{\delta})$, Proposition \ref{Schwarz} implies that $\|g(x)\|\le\|\delta(x)\|$ for any $x\in B_{\delta}$. Therefore, we have
\begin{equation*}
\|\phi_{\alpha}(f(x_0))\|\le\|\delta(x_0)\|<1.
\end{equation*}
On the other hand, since $\|f(x_0)\|=1$, we have $\|\phi_{\alpha}(f(x_0))\|=1$, a contradiction. Hence $f$ must be a constant nc function.
\end{proof}


\begin{remark}
\upshape Salomon, Shalit and Shamovich \cite[Lemma 6.11]{SSS} has already proved a noncommutative analog of maximum principle in a more general setting. However, our proof is quite different from theirs, and still works even in an infinite dimensional setting like \cite[chapter 16]{AMY} and \cite{AuM}. (We do not know how to apply their proof to the infinite dimensional setting.)
\end{remark}


\begin{remark}
\upshape Popescu \cite[Theorem 2.7 and Theorem 2.8]{Pop2016} also proved noncommutative analogs of Schwarz lemma and the maximum principle for free holomorphic functions on the regular polyball. With $d=d_1+\cdots d_k$, we define the noncommutative polyball $\mathbf{P_d}^{nc}\subset\mathbb{M}^d$ as the polynomial polyhedron associated with
\begin{equation*}
\delta(x):=
\begin{bmatrix}
[x_{1,1}\cdots x_{1,d_1}]& & \\
 &\ddots& \\
 & & [x_{k,1}\cdots x_{k,d_k}]
\end{bmatrix}.
\end{equation*}

Next, we define the polyball $\mathbf{P_d}\subset\mathbf{P_d}^{nc}$ to be all tuples $x=(x_1,\ldots,x_k)\in\mathbf{P_d}^{nc}$ with the property that the entries of $x_s:=(x_{s,1},\ldots, x_{s, d_s})$ commute with the entries of 
$x_t:=(x_{t,1},\ldots, x_{t, d_t})$ for any distinct $s$, $t\in\{1,\ldots k\}$. The regular polyball of $\mathbb{M}^d$ is defined by
\begin{equation*}
\mathbf{B_d}:=\{x\in\mathbf{P_d}\;|\; \Delta_x (I) \;\mbox{is strictly positive definite}\}, 
\end{equation*}
where for any $x\in(\mathbf{P_d})_n$, $\Delta_x:\mathbb{M}_n\rightarrow\mathbb{M}_n$ is given by
\begin{equation*}
\Delta_x:=(id-\Phi_1)\circ\cdots\circ (id-\Phi_k)
\end{equation*}
and $\Phi_i:\mathbb{M}_n\rightarrow\mathbb{M}_n$ is the completely positive linear map defined by 
\begin{equation*}
\Phi_i(y):=\sum_{j=1}^{d_i} x_{i,j} y x_{i,j}^*.
\end{equation*}
Since every free polynomial is nc, 
\begin{equation*}
\mathbf{P_d}=\{x\in\mathbf{P_d}^{nc}\;|\;x_{s,i}x_{t,j}-x_{t,j}x_{s,i}=0\;\;(1\le s\ne t\le k, 1\le i\le d_s, 1\le j\le d_t)\}
\end{equation*}
equals its $B_{\delta}$-relative full nc envelope 
$[\mathbf{P_d}]_{\mathrm{full}}\cap \mathbf{P_d}^{nc}$ (see \cite[Definition 2.8 and Proposition 2.9]{BMV}). Therefore, by \cite[Corollary 3.4]{BMV}, every bounded nc function on the polyball $\mathbf{P_d}$ can be extended to a bounded nc function on the noncommutative polyball $\mathbf{P_d}^{nc}$ without increasing its operator norm.  Hence we can prove an nc analog of Schwarz lemma and the maximum principle for bounded nc functions on $\mathbf{P_d}$. However, the author cannot treat the regular polyball $\mathbf{B_d}$ with our methods based on the nc Schur-Agler class and its realization formula.

\end{remark}


At the end of this section, we will show that the Cayley transforms between the open unit disk and the right half plane defined by
\begin{equation*}
z=\frac{1+x}{1-x}\;\;\;(x\in\mathbb{D}),\;\;\;x=\frac{z-1}{z+1}\;\;\;(\Re{z}>0)
\end{equation*}
give a homeomorphism between $\mathcal{RSA}(B_{\delta})$ and $\mathrm{RHA}(B_{\delta})$. In the course of proving this fact, we prove the next analog of Schwarz lemma for $\mathrm{RHA}(B_{\delta})$ as a consequence of Proposition \ref{Schwarz}.

\begin{prop}\label{corSchwarz}
For any $h\in\mathrm{RHA}(B_{\delta})$,
\begin{enumerate}
\item $\displaystyle\Re h(x)\ge\frac{1-\|\delta(x)\|}{1+\|\delta(x)\|}I\;\;\;(x\in B_{\delta})$,
\item $\displaystyle\|h(x)\|\le\frac{1+\|\delta(x)\|}{1-\|\delta(x)\|}\;\;\;(x\in B_{\delta})$.
\end{enumerate}
\end{prop}
\begin{proof}
Let $f(x):=(h(x)-I)(h(x)+I)^{-1}\in\mathcal{RSA}(B_{\delta})$. By Proposition \ref{Schwarz}, $\|f(x)\|\le\|\delta(x)\|$. Note that $h(x)=(I+f(x))(I-f(x))^{-1}$. Item (2) immediately follows from this expression with the aid of the Neumann series. More precisely, we have
\begin{equation*}
\|(I-f(x))^{-1}\|\le\sum_{k=0}^{\infty}\|f(x)\|^k\le\sum_{k=0}^{\infty}\|\delta(x)\|^k=\frac{1}{1-\|\delta(x)\|}.
\end{equation*}
 Moreover, 
\begin{align*}
2\Re h(x)
&=(I+f(x))(I-f(x))^{-1}+(I-f(x)^*)^{-1}(I+f(x)^*)\\
&=(I-f(x)^*)^{-1}[(I-f(x)^*)(I+f(x))+(I+f(x)^*)(I-f(x))](I-f(x))^{-1}\\
&=2(I-f(x)^*)^{-1}(I-f(x)^*f(x))(I-f(x))^{-1}\\
&\ge2(1-\|\delta(x)\|^2)(I-f(x)^*)^{-1}(I-f(x))^{-1}.
\end{align*}
Since $\|I-f(x)\|\le1+\|\delta(x)\|$, we have $(1+\|\delta(x)\|)^2I\ge(I-f(x))(I-f(x)^*)$ and hence, $\displaystyle\frac{1}{(1+\|\delta(x)\|)^2}I\le(I-f(x)^*)^{-1}(I-f(x))^{-1}$. So, 
$\displaystyle\Re h(x)\ge\frac{1-\|\delta(x)\|}{1+\|\delta(x)\|}I$.
\end{proof}


\begin{remark}
\upshape One can also prove the above inequalities by using \cite[Lemma 3.3]{PPT}.
\end{remark}


\begin{prop}\label{homeo}
The Cayley transforms $f\leftrightarrow h$ between $\mathcal{RSA}(B_{\delta})$ and $\mathrm{RHA}(B_{\delta})$ defined by
\begin{align*}
h(x)&:=(I+f(x))(I-f(x))^{-1}\;\;\;(f\in\mathcal{RSA}(B_{\delta})),\\
f(x)&:=(h(x)-I)(h(x)-I)^{-1}\;\;\;(h\in\mathrm{RHA}(B_{\delta}))
\end{align*}
are homeomorphisms.
\end{prop}
\begin{proof}
It is sufficient to prove that these maps are sequentially continuous.

Suppose that a sequence $\{f_n\}_{n=1}^{\infty}$ converges to $f$ in $\mathcal{RSA}(B_{\delta})$ with the topology of uniform convergence on closed polynomial polyhedra $K_{\delta, r}$. By the resolvent identity and Proposition \ref{Schwarz}, we have
\begin{align*}
&\|(I+f_n(x))(I-f_n(x))^{-1}-(I+f(x))(I-f(x))^{-1}\|\\
&\le \|(I+f_n(x))(I-f_n(x))^{-1}-(I+f_n(x))(I-f(x))^{-1}\|\\
&\;\;\;+\|(I+f_n(x))(I-f(x))^{-1}-(I+f(x))(I-f(x))^{-1}\|\\
&\le 2\|(I-f_n(x))^{-1}(f_n(x)-f(x))(I-f(x))^{-1}\|+\|f_n(x)-f(x)\|\|(I-f(x))^{-1}\|\\
& \le3\frac{1}{(1-\|\delta(x)\|)^2}\|f_n(x)-f(x)\|.
\end{align*}
Therefore, $(I+f_n(x))(I-f_n(x))^{-1}$ converges to $(I+f(x))(I-f(x))^{-1}$ uniformly on each $K_{\delta, r}$.

Let us prove that the inverse mapping is also continuous. Suppose that a sequence $\{h_n\}_{n=1}^{\infty}$ converges to $h$ in $\mathrm{RHA}(B_{\delta})$. Note that for any functions $H\in\mathrm{RHA}(B_{\delta})$, inequality (1) in Proposition \ref{corSchwarz} implies
\begin{equation*}
\|(H(x)+I)^{-1}\|\le\sqrt{\frac{1+\|\delta(x)\|}{2(1-\|\delta(x)\|)}}<\sqrt{\frac{1}{1-\|\delta(x)\|}}
\end{equation*}
by considering $(H(x)+I)(H(x)^*+I)$. Therefore, by the resolvent identity and inequality (2) in Proposition \ref{corSchwarz}, we have
\begin{align*}
&\|(h_n(x)-I)(h_n(x)+I)^{-1}-(h(x)-I)(h(x)+I)^{-1}\|\\
&\le\sqrt{\frac{1}{(1-\|\delta(x)\|)}}\|h_n(x)-h(x)\|+\frac{2}{(1-\|\delta(x)\|)^2}\|h_n(x)-h(x)\|.
\end{align*}
So, $(h_n(x)-I)(h_n(x)+I)^{-1}$ converges to $(h(x)-I)(h(x)+I)^{-1}$ uniformly on each $K_{\delta, r}$.
\end{proof}


Pascoe, Passer and Tully-Doyle \cite[Proposition 2.2]{PPT} showed that the free Herglotz-Agler class is compact in the pointwise convergence topology in the infinite dimensional setting. In our finite dimensional setting, we can prove that $\mathrm{RHA}(B_{\delta})$ is compact in the topology defined by a certain uniform convergence. 

The {\bf disjoint union (du) topology} on $\mathbb{M}^d$ is the topology consisting of all the sets $\Omega$ such that each section $\Omega_n$ is open in the Euclidean topology on $\mathbb{M}_n^d$. Here, we consider the uniform convergence on du compact sets. Note that a du compact set is free compact. Then, \cite[Proposition 4.14]{AMg} and Proposition \ref{homeo} imply the following result:

\begin{prop}
Let $\{h_n\}_{n=1}^{\infty}$ be a sequence in $\mathrm{RHA}(B_{\delta})$. Then, there exists a subsequence $\{h_{n_k}\}$ and $h\in\mathrm{RHA}(B_{\delta})$ such that $\{h_{n_k}\}$ converges to h uniformly on each du compact subset of $B_{\delta}$.
\end{prop}


\section{Polynomial approximation theorem}

In this final section, we will give a polynomial approximation type characterization of regular free Herglotz-Agler functions. We have seen that $\mathcal{RSA}(B_{\delta})$ and $\mathrm{RHA}(B_{\delta})$ are homeomorphic (Proposition \ref{homeo}). We need the following fact:

\begin{theorem}$($\cite[Theorem 3.3]{Koj}$)$\label{Kojin}
A graded function $f$  on $B_{\delta}$ belongs to $\mathcal{SA}(B_{\delta})$ if and only if there exists a sequence of free polynomials $\{p_n\}_{n=1}^{\infty}$ such that $p_n$ converges to $f$ uniformly on each free compact subset of $B_{\delta}$, and the norm of $p_n$ is uniformly less than one.
\end{theorem}


\begin{remark}
$(1)$ \upshape The sequence in Theorem \ref{Kojin} apparently converges on each $K_{\delta,r}$. It is appropriate to consider this convergence because every closed polynomial polyhedron is not free compact (see the discussion below \cite[Corollary 3.2]{Koj}).\\
$(2)$ By construction, if $\delta(0)=0$ and $f(0)=0$, then the polynomials $p_n$ in Theorem \ref{Kojin} must satisfy $p_n(0)=0$.
\end{remark}

Theorem \ref{Kojin} implies the following lemma:


\begin{lemma}\label{lem}
For any $h\in\mathrm{RHA}(B_{\delta})$, there exists a sequence of free polynomials $\{p_n\}_{n=1}^{\infty}$ such that the restriction of each $p_n$ to $B_{\delta}$ is in $\mathcal{RSA}(B_{\delta})$ and $(1+p_n)(1-p_n)^{-1}$ converges to $h$ uniformly on each $K_{\delta, r}$. 
\end{lemma}

\begin{proof}
Define a regular nc Schur-Agler function $f$ by $f(x):=(h(x)-I)(h(x)+I)^{-1}$. Then, this lemma can easily be proved by Proposition \ref{homeo} and Theorem \ref{Kojin} .
\end{proof}


The next result gives a polynomial approximation type characterization of regular free Herglotz-Agler functions. 


\begin{theorem}
A graded function $h$ on $B_{\delta}$ belongs to $\mathrm{RHA}(B_{\delta})$ if and only if there exists a sequence of free polynomials $\{p_n\}_{n=1}^{\infty}$ such that $p_n|_{B_{\delta}}\in \mathrm{RHA}(B_{\delta})$ and $p_n$ converges to $h$ uniformly on every $K_{\delta, r}$.
\end{theorem}
\begin{proof}
First, we prove that for any $f\in\mathcal{RSA}(B_{\delta})$ and $0<r<1$, there exists an $N\in\mathbb{N}$ such that if $n\ge N$, then
\begin{equation*}
\Re(I+rf(x))\sum_{k=0}^n(rf(x))^k\ge0\;\;\;(x\in B_{\delta}).
\end{equation*}
In the same way as in the proof of Proposition \ref{corSchwarz}, we obtain that
\begin{align*}
&2\Re (I+rf(x))(I-rf(x))^{-1}\\
&=(I+rf(x))(I-rf(x))^{-1}+(I-rf(x)^*)^{-1}(I+rf(x)^*)\\
&=(I-rf(x)^*)^{-1}[(I-rf(x)^*)(I+rf(x))+(I+rf(x)^*)(I-rf(x))](I-rf(x))^{-1}\\
&=2(I-rf(x)^*)^{-1}(I-r^2f(x)^*f(x))(I-rf(x))^{-1}\\
&\ge2(1-r^2)(I-rf(x)^*)^{-1}(I-rf(x))^{-1}.
\end{align*}
Since $\|I-rf(x)\|\le 2$, we have
\begin{equation*}
\Re(I+rf(x))(I-rf(x))^{-1}\ge\frac{1-r^2}{4}I.
\end{equation*}
We may and do assume that the natural number $N$ (by replacing it with a larger one if necessary) also satisfies that if $n\ge N$, then $\displaystyle\sum_{k=n+1}^{\infty} r^k<\frac{1-r^2}{8}$. Using the Neumann series expansion, we have if $n\ge N$, 
\begin{align*}
&\|\Re (I+rf(x)) \sum_{k=0}^n (rf(x))^k-\Re (I+rf(x))(I-rf(x))^{-1}\|\\&\le 2\sum_{k=n+1}^{\infty} r^k<\frac{1-r^2}{4}.
\end{align*}
Therefore, for every $n\ge N$, we obtain that
\begin{align*}
&\Re (I+rf(x))\sum_{k=0}^n(rf(x))^k\\
&=\left(\Re (I+rf(x))\sum_{k=0}^n (rf(x))^k-\Re (I+rf(x))(I-rf(x))^{-1}\right)\\
&\;\;\;+\Re (I+rf(x))(I-rf(x))^{-1}\\
&\ge -\frac{1-r^2}{4}I+\frac{1-r^2}{4}I\ge 0.
\end{align*}

Next, we choose $0<r_n<1$ so that $r_n$ converges to $1$ incresingly, and choose $p_n$ as in Lemma \ref{lem}. We have seen that for any $n\in\mathbb{N}$, there exists a $K_n\in\mathbb{N}$ such that if $k\ge K_n$, then 
\begin{equation*}
\Re (I+r_np_n(x))\sum_{j=0}^k(r_np_n(x))^j\ge 0.
\end{equation*}
Set $\displaystyle q_n(x):=(I+r_np_n(x))\sum_{j=0}^{L_n}(r_np_n(x))^j$, where
 $L_n=\mathrm{max}\{K_n, L_{n-1}\}+1$. Then, $q_n\in\mathrm{RHA}(B_{\delta})$. We can also prove that $q_n$ converges to $h$ uniformly on each $K_{\delta, r}$ as an application of the techniques used in this paper. By the resolvent identity, we have
 \begin{align*}
 &\|q_n(x)-h(x)\|\\
 &\le\|(I+r_np_n(x))\sum_{j=0}^{L_n}(r_np_n(x))^j-(I+p_n(x))\sum_{j=0}^{L_n}(r_np_n(x))^j\|\\
 &\;\;\;+\|(I+p_n(x))\sum_{j=0}^{L_n}(r_np_n(x))^j-(I+p_n(x))(I-r_np_n(x))^{-1}\|\\
 &\;\;\;+\|(I+p_n(x))(I-r_np_n(x))^{-1}-(I+p_n(x))(I-p_n(x))^{-1}\|\\
 &\;\;\;+\|(I+p_n(x))(I-p_n(x))^{-1}-h(x)\|\\
 &\le \|(1-r_n)p_n(x)\|\sum_{j=0}^{L_n}\|p_n(x)\|^j+\|I+p_n(x)\|\sum_{j=L_n+1}^{\infty}\|r_np_n(x)\|^j\\
 &\;\;\;+\|I+p_n(x)\|\|(I-r_np_n(x))^{-1}(r_np_n(x)-p_n(x))(I-p_n(x))^{-1}\|\\
 &\;\;\;+\|(I+p_n(x))(I-p_n(x))^{-1}-h(x)\|.
 \end{align*}
 Hence, Proposition \ref{Schwarz} and the Neumann series expansion implies that 
 \begin{align*}
 &\|q_n(x)-h(x)\|\le |1-r_n|\sum_{j=0}^{L_n}\|\delta(x)\|^j+2\sum_{j=L_n+1}^{\infty}\|\delta(x)\|^j\\
 &\;\;\;+2\frac{1}{(1-\|\delta(x)\|)^2}|r_n-1|+\|(I+p_n(x))(I-p_n(x))^{-1}-h(x)\|.
 \end{align*}
Therefore, we conclude that $q_n$ converges to $h$ uniformly on each $K_{\delta, r}$.
 
 The converse direction is trivial.
\end{proof}


\section*{Acknowledgment}

The author acknowledges his supervisor Professor Yoshimichi Ueda for his encouragements. He also acknowledges Professor John Edward McCarthy for some comments from a viewpoint of experts.


\begin{thebibliography}{8}
\bibitem{Agl} J. Agler, On the representation of certain holomorphic functions defined on a polydisk, {\it Operator Theory: Advances and Applications}, vol.48, 47-66, Birkh$\ddot{\mathrm{a}}$user, Basel, 1990.
\bibitem{AMg} J. Agler and J. E. McCarthy, Global holomorphic functions in several non-commuting variables, {\it Canadian J. Math.}, 67(2):241-285, 2015
\bibitem{AMp} J. Agler and J. E. McCarthy, {\it Pick interpolation and Hilbert Function Spaces}, Graduate Studies in Mathematics, vol. 44, American Mathematical Society, Providence, RI, 2002.
\bibitem{AMY}J. Agler,  J. E. McCarthy, and N. Young, {\it Operator Analysis Hilbert Space Methods in Complex Analysis}, Cambridge University Press, 2020.
\bibitem{AT}C.-G. Ambrozie and D. Timotin, A von Neumann type inequality for certain domains
in $\mathbb{C}^d$,  {\it Proc. Amer. Math. Soc.}, 131:859-869, 2003.
\bibitem{ADR} J. M. Anderson, M. A. Dritschel and J. Rovnyak, Schwarz-Pick inequalities for the Schur-Agler class on the polydisk and unit ball. {\it Comput. Methods Funct. Theory 8}, 2008, no. 1-2, 339-361.
\bibitem{AuM} M. Augat and J. E. McCarthy, Operator NC functions, {\it Canadian Mathematical Bulletin}, to appear.
\bibitem{BMV}J. A. Ball, G. Marx and V. Vinnikov, Interpolation and transfer function realization for the noncommutative Schur-Agler class. In {\it Operator Theory in Different Setting and Related Applications}. {\it Operator Theory:Advances and Applications}, vol. 262, pages 23-116. Birkh$\ddot{\mathrm{a}}$user/Springer, Cham, Switzerland, 2018.
\bibitem{Her} G. Herglotz, $\ddot{\mathrm{U}}$ber Potenzreihen mit positivem, rellen Teil im Einheitskreis, {\it Ber. Verh. Sachs. Akad. Wiss. Leipzig}, 63:501-511, 1911.
\bibitem{KVV} D. S. Kaliuzhnyi-Verbovetskyi and V. Vinnikov, {\it Foundations of Noncommutative Function Theory}, Mathematical Surveys and Monographs, vol. 199, American Mathematical Society, Providence, RI, 2014.
\bibitem{Kne} G. E. Knese, A Schwarz lemma on the polydisk. {\it Proc. Amer. Math. Soc.}, 135, no. 9, 2759-2768, 2007.
\bibitem{Koj} K. Kojin, A refined nc Oka-Weil theorem, {\it Canadian Mathematical Bulletin}, to appear.
\bibitem{PPT} J. E. Pascoe, B. Passer and R. Tully-Doyle, Representation of free Herglotz functions, {\it  Indiana Univ. Math. J.} 68:4, 199-1215, 2019.
\bibitem{Pop2006} G. Popescu, Free holomorphic functions on the unit ball of $B(\mathcal{H})$, {\it J. Funct. Anal.} 241, 268-333, 2006.
\bibitem{Pop2010} G. Popescu, Free holomorphic automorphisms of the unit ball of $B(\mathcal{H})$, {\it J. Reine Angew. Math.} 638, 119-168, 2010.
\bibitem{Pop2016} G. Popescu, Holomorphic automorphisms of noncommutative polyballs, {\it J. Operator Theory} 76, no. 2, 387-448, 2016.
\bibitem{Rud} W. Rudin, {\it Real and Complex Analysis Third Edition}, McGraw-Hill, Singapore, 1987
\bibitem{SSS} G. Salomon, O. M. Shalit and E. Shamovich, Algebras of bounded noncommutative analytic functions on subvarieties of the noncommutative unit ball, {\it Trans. Am. Math. Soc.} 370(12), 8639-8690, 2018.

\end{thebibliography}
\end{document}